\documentclass[18pt]{amsart}
\usepackage{hyperref}
\usepackage[usenames]{color}
\usepackage{amsmath,amsthm,amsfonts,amssymb}

\newtheorem{theorem}{Theorem}
\newtheorem{lemma}[theorem]{Lemma}
\newtheorem{corollary}[theorem]{Corollary}

\theoremstyle{definition}

\newtheorem{remark}[theorem]{Remark}

\numberwithin{equation}{section}

\begin{document}

\title{Embedding theorems for  solvable groups}

\author{Vitaly Roman'kov
}

\maketitle
\footnote{The work is supported by Mathematical Center in Akademgorodok under agreement No. 075-15-2019-1613 with the Ministry of Science and Higher Education of the Russian Federation.}
\begin{abstract}
In this paper, we prove a series of results on group embeddings in groups with a small number of generators.  We show that each finitely generated group $G$  lying in a variety ${\mathcal M}$ can be embedded in a $4$-generated group $H \in {\mathcal M}{\mathcal A}$ (${\mathcal A}$ means the variety of abelian groups).  If $G$ is a finite  group, then $H$ can also be found as a finite   group. 
It follows,  that any finitely generated (finite) solvable group $G$ of the derived length $l$ can be embedded in a $4$-generated (finite) solvable group $H$ of length $l+1$. Thus, we  answer  the question of V. H. Mikaelian and A.Yu. Olshanskii. It is also shown that any countable group  $G\in {\mathcal M}$,  such that the  abelianization $G_{ab}$ is a free abelian group,  is embeddable in a  $2$-generated group $H\in {\mathcal M}{\mathcal A}$.  
\end{abstract}



\section{Introduction}

The main aim of this paper is to study embeddings of finitely generated groups in $2$- or $4$-generated groups. Let $G$ stand for a finitely generated group which lies in a variety ${\mathcal M}$, and $H$ for a $4$-generated group in which $G$ can be embedded. We show that $H$ can be found in the variety ${\mathcal M}{\mathcal A}$, where ${\mathcal A}$ denote the variety of abelian groups. It follows that every finitely generated  solvable group of the derived length $l$ can be embedded in a $4$-generated solvable group of length $l+1$. 

We also study what further properties of $G$ our embedding procedure endow $H.$ Finiteness is one of them, and if $G$ is a finite $p$-group ($p$ is a prime) , $H$ can be chosen as a finite $p$-group. Also, if $G$ has finite exponent, then $H$ can be made to have finite exponent. If $G$ is a counable group such that the  abelianization $G_{ab}$ is a free abelian group then $H$ can be found as a $2$-generated group. 

Thus, we refine the classical results on embeddings of countable groups, a brief overview of which is given below. These results refer to embeddings in $2$-generated groups. We do not know whether parameter $4$ can be lowered in our results. We  believe that this cannot be done.

In the late 1940s, G. Higman, B. H. Neumann and H. Neumann showed that every countable group embeds in a $2$-generator group, in the same paper \cite{HNN} in which they introduced and succesifully applied HNN-extensions. Their method using free constructions of groups did not give similar results for varieties of groups. So then  B. H. Neumann and H. Neumann in \cite{NN} applied wreath products to   prove that every countable group $G$ lying in a variety ${\mathcal M}$ can be embedded in a $2$-generated group $H \in {\mathcal M}{\mathcal A}^2.$ If $G$ is finite ($p$-) group, then $H$ can be chosen as finite ($p$-) group.  Also, if $G$ has finite exponent, then $H$ can be made to have finite exponent. 

It follows that every countable solvable group of the derived length $l$ can be embedded  in some $2$-generated solvable group of length $l+2.$ Note that this bound is sharp. Namely, the group $\mathbb{Q}$ of rationals does not embed into any finitely generated metabelian group $M.$ Indeed, $M$ is residually finite by   Hall$'$s theorem proved in \cite{Hall}, but $\mathbb{Q}$ is not.  Thus we cannot lower $l+2$ to $l+1$ in the Neumann-Neumann embedding theorem.

V.H. Mikaelian and A. Yu. Olshanskii   gave an explicit classification of all abelian groups that can occur as subgroups of finitely generated metabelian groups as follows.

{\bf Theorem} (V. H. Mikaelian, A. Yu. Olshanskii \cite{MO}).
Let $A$ be an abelian group. The following properties are equivalent.
\begin{enumerate}
\item $A$ is a subgroup of a finitely generated metabelian group;
\item  $A$ is a subgroup of a finitely generated abelian-by-polycyclic group;
\item  $A$ is a subgroup of a finitely presented metabelian group;
\item $A$ is a subgroup of a $2$-generated metabelian group;
\item $A$ is a Hall group.
\end{enumerate}

Note that (3) follows from  a remarkable statement independently proved by G. Baumslag and V.N. Remeslennikov: each finitely generated metabelian group embeds in some finitely presented metabelian group (see \cite{Romess}).
 
By definition, $A$ is a {\it Hall group} if
\begin{itemize}
\item
 $A$ is a (finite or) countable abelian group;
\item $A = T \oplus K,$ where $T$ is a bounded torsion group (i.e., the orders of all
elements in $T$ are bounded), $K$ is torsion free;
\item $K$ has a free abelian subgroup $F$ such that $K/$F is a torsion group with
trivial $p$-subgroups for all primes except for the members of a finite set defined by $K.$
\end{itemize} 

In \cite{O}, A. Yu. Olshanskii established a number of other embedding theorems for metabelian groups. 

Every finitely generated nilpotent group satisfies  the maximal  condition on subgroups, that is, every subgroup is finitely generated (see \cite{Hall2} or \cite{Romess}). Hence each non-finitely generated  nilpotent group cannot be embedded in a finitely generated nilpotent group. However  every finitely generated nilpotent group  embeds in some $2$-generated nilpotent group of sufficiently large class \cite{Romnilp}. Similarly, every polycyclic group  embeds in a $2$-generated polycyclic group \cite{Rompol}. 

Let $G$ be a group and $g,f\in G.$ Further in the paper $g^f$ denotes $f^{-1}gf$ (conjugate of $g$ by $f$) and $[g,f]$ stands for $g^{-1}f^{-1}gf$ (commutator of $g$ and $f$). Also  $[g, f; 1] $  means $[g,f]$ and inductively $[g,f; k+1]$ stands for $[[g,f;k],f], k = 1, 2, ... .$
 By $G'$we denote  the derived subgroup of $G$. Then $G_{ab}=G/G'$ is the abelianization of $G.$ $\mathbb{Z}$ means the infinite cyclic group and $\mathbb{Z}_n$ denotes a cyclic group of order $n.$

Recall that the Cartesian wreath product of groups is defined as follows. 
Let $A$ and $B$ be groups and $D$ a group of all functions $f : B \rightarrow A$ with multiplication $(f_1f_2)(x) = f_1(x)f_2(x)$ for $x \in B.$  The group $B$ acts on $D$
from the left by shift automorphisms: $f^b(x) = f(bx)$ for all $f \in  D,$
$b, x \in B,$ and the associated with this action semidirect product $D \rtimes B$ is
called the {\it Cartesian  wreath product} of the groups $A$ and $B,$ denoted by
$A Wr B.$ The subgroup $D$ is called {\it base subgroup} of  $A\ Wr\ B$. Thus, every element of $A Wr B$ has a unique presentation as $bf$
$(b \in  B, f \in D)$ and the multiplication rule follows from the conjugation
formula 
\begin{equation}
\label{eq:1.1}
 f^b(x)=f(bx) 
\end{equation}
\noindent
in $A Wr B$ for any $b, x \in B$ and $f \in D.$ If instead of $D$
one takes the smaller group 
consisting of all functions with finite support, that is, functions taking only non-identity values on a finite set of points, then one obtains a subgroup of $A\ Wr\ B$
called the {\it wreath product} ({\it direct wreath product}); it is denoted by 
$A\ wr\ B$.

\section{Main results}

The following question was posed by V. H. Mikaelian and A. Yu. Olshanskii in \cite{MO} and  was also written by A. Yu. Olshanskii  in \cite{Kour}(Question 18.73): 

Does every finitely generated solvable group  of derived length $l \geq 2$ embed
into a $2$-generated solvable group  of length $l + 1${}? Or at least, into some $k$-generated
$(l + 1)$-solvable group, where $k = k(l)${}? 

We prove the following embedding theorems that imply an affirmative answer to  Mikaelian-Olshanskii$'$ question.    In the following statements, ${\mathcal M}$  means an arbitrary variety of groups and ${\mathcal A}$ is the variety of abelian groups. For $s\in \mathbb{N}$, ${\mathcal A}_s$ means the variety of abelian groups of exponent $s.$  
\begin{theorem}
\label{th:1}
Let  $G$ be a countable  group  such that the abelianization $G_{ab}$ is a free abelian group.  Then $G$  embeds in some   $2$-generated subgroup  $H$ of the Cartesian wreath product $G Wr \mathbb{Z}$. 
\end{theorem}
\begin{corollary}
\label{co:2}
 \begin{enumerate}
 \item Let $G \in {\mathcal M}$  be a countable group such that the abelianization $G_{ab}$ is a free abelian group.  Then $G$  embeds in some $2$-generated  group $H \in {\mathcal MA}.$ In particular, every finitely generated  group $G\in {\mathcal M}$  such that the abelianization $G_{ab}$ is torsion-free, embeds in some $2$-generated  group $H \in {\mathcal MA}.$
\item Let  $G$ be a countable solvable group of  derived length  $l$ such that the abelianization $G_{ab}$ is a free abelian group.  Then $G$  embeds in some $2$-generated solvable group $H$ of length $l+1.$ 
In particular, every finitely generated solvable group $G$ of  derived length $l$ such that the abelianization $G_{ab}$ is torsion-free, embeds in some $2$-generated solvable group $H$ of length $l+1.$
\item Every finitely generated   group $G \in {\mathcal M}$ has a subgroup $K$ of  finite index that   can be embedded in some $2$-generated group $H \in {\mathcal MA}.$ 
In particular, every finitely generated solvable group $G$
of  derived length  $l$ has a subgroup $K$ of  finite index that  can be embedded in some $2$-generated solvable group $H$ of length $l+1.$
\end{enumerate}
\end{corollary}
\begin{theorem}
\label{th:3}
Let  $G$ be a countable  group  such that  the abelianization $G_{ab}$  is a direct product of a free abelian group and a finite group.  Then $G$  embeds in some $4$-generated subgroup $H$ of   the Cartesian wreath product $G\ Wr \  \mathbb{Z}^3 $. 
\end{theorem}
\begin{corollary}
\label{co:4}
\begin{enumerate}
\item 
 Let $G \in {\mathcal M}$  be a countable group such that the abelianization $G_{ab}$ is a direct product of a free abelian group and a finite group.  
 Then $G$  embeds in a $4$-generated subgroup $H\in {\mathcal MA}.$ In particular, every finitely generated  group $G\in {\mathcal M}$   embeds in some $4$-generated  group $H \in {\mathcal MA}.$
\item 
Let  $G$ be a countable  solvable group  of  derived length $l$ such that the  abelianization $G_{ab}$ is  a direct product of a free abelian group and a finite group.  Then $G$  embeds in a $4$-generated solvable group $H$ of  length $l+1.$ In particular, 
every finitely generated solvable group $G$ of  derived length $l$  embeds in some $4$-generated solvable group $H$ of length $l+1.$
\end{enumerate}
\end{corollary}

\begin{theorem}
\label{th:5}
Let $G$  be a group generated by a finite set $u_1, ..., u_m$ of elements of finite orders $l_1, ..., l_m$, respectively.   Then $G$  embeds in some $4$-generated  subgroup $H$ of $\tilde{W}_s = G\ wr\  (\mathbb{Z}_{s^2} \times \mathbb{Z}_s^2)$ where $s = $ {\rm lcm(}$ l_1^{m}, ..., l_m^{m}${\rm )}. 
\end{theorem} 

\begin{corollary}
\label{co:6}
\begin{enumerate}
\item 
 Let $G\in {\mathcal M}$ be a $m$-generated group of exponent $e$.  
 Then $G$  embeds in a $4$-generated group $H\in {\mathcal M}{\mathcal A}_s$ where $s = e^{m+1}$, and so has exponent $e^{m+2}.$
 \item
Let $G\in {\mathcal M}$ be a finite group. Then $G$ embeds in some $4$-generated finite group $H\in {\mathcal MA}$. In particular, 
every finite  solvable {\rm (}$p$- {\rm )} group $G$ of  derived length $l$  embeds in some $4$-generated finite solvable {\rm (}$p$- {\rm )} group $H$ of length $l+1.$ 
\end{enumerate}
\end{corollary}

\section{Proof of   Theorem \ref{th:1}.}
Let $G$ be a countable group such that $\bar{A} = G_{ab}$ is a  free abelian  group with basis $\{\bar{a}_i : i \in I \subseteq \mathbb{N}\}$. Denote by $a_i$ a preimage of $\bar{a}_i$ in $G$ and set $A=$ gp($a_i: i \in I$).

 Let $C$ be an infinite cyclic group generated by $c$  and $U=G\ Wr\ C \simeq  G\ Wr\ \mathbb{Z}.$  We are to show that $G$ embeds in some $2$-generator subgroup $H$ of $U$, and thus  prove  Theorem \ref{th:1}. Now $D_U$ denotes the base group of $U.$ 

 Let $s_1 < s_2 < ... < s_i < ... $ be a sequence of positive integers such that $s_{i+j} -s_j=s_{k+l}-s_l$ if and only if $i=k$ and $j=l.$ For definiteness, we  take  $s_{i}=  2^i$ for $i= 1, ... .$ This sequence is called a {\it strictly  uneven
sparse sequence}.  

For brevity,  denote $c_i=c^{s_i}, i = 1, 2, ... ,$  and  set  
${\mathcal C} = \{c_i: i = 1, 2, ...\}.$

Suppose that $G$ is generated by a set of generating elements $\{ g_j: j\in J\subset \mathbb{N}\}.$ Let  $g_j = a(j) g_j',$ where $a(j) \in A,$ $g_j'\in G', \ j \in J.$ Then $G=$ gp($a_i, g_j': i\in I, j\in J$).

Let
\begin{equation}
\label{eq:3.1}
g_j' = \prod_{q=1}^{r_j}[u_{j,q}, v_{j,q}];  u_{j,q}, v_{j,q}\in G;   j \in J.
\end{equation}

Let $H=$ gp($c, d$) where $d\in D_U$ is defined as follows. Let ${\mathcal C}_1\subseteq {\mathcal C}, {\mathcal C}_1 = {\mathcal A} \sqcup {\mathcal U} \sqcup {\mathcal V}\  (|{\mathcal A}|=|I|,  |{\mathcal U}|= |{\mathcal V}|=|\{u_{j,q}: q = 1, ..., r_j,  j\in J\}|=|\{v_{j,q}: q = 1, ..., r_j,  j\in J\}|)$ is a disjoint union of sets.  Here  $I$ is in one-to-one correspondence $\iota$ with indexes of elements in ${\mathcal A}$, the set of all elements of the form $u_{j,q}$ is in one-to-one correspondence $\delta$ with the set of  indexes of all elements in ${\mathcal U}$, and  the set of all elements of the form $v_{j,q}$  is in one-to-one correspondence $\lambda$ with the set of indexes of all elements in ${\mathcal V}$. 

Then we set 
$$d(c_{\iota (i)}) = a_i \  {\mathrm{for\ each}} \ i \in I;$$
$$ d(c_{\delta (u_{j,q})}) = u_{j,q}, \  \  \ {\mathrm{for \  each} \   u_{j,q}};$$ \begin{equation}
\label{eq:3.2}
d(c_{\lambda (v_{j,q})})  = v_{j,q}  \   \ {\mathrm{for \  each} \   v_{j,q}},
\end{equation} 
\noindent 
and  $d(c^s) = 1$ in all other cases.   

First, we prove that $H$ contains all the elements  $\tilde{g}_j' \in D_U$ such that $\tilde{g}_j'(1) =  g_j'$ and $ \tilde{g}_j'(c^s) = 1$ for $ s \not=0$. For any pair $(j, q),$  we compute  by direct computation that
\begin{equation}
\label{eq:3.3}
[d^{c_{\delta (u_{j,q})}},  d^{c_{\lambda (v_{j,q})}}](1) = [u_{j,q}, v_{j,q}].
\end{equation}
 
We set $\tilde{g}_j'= [d^{c_{\delta (u_{j,q})}},  d^{c_{\lambda (v_{j,q})}}].$ It remains to verify that every value $\tilde{g}_j'(c^s)$ is trivial for each $s\not=0.$ This statement follows  because the sequence $s_1, s_2, ...,$ is strictly uneven
sparse, and therefore  each other non-trivial value $d^{c_{\kappa (u_{j,q})}}(c^s), s \not= 0,$  meets  the trivial value of $d^{c_{\lambda (v_{j,q})}}(c^s)$. Then we get, by (\ref{eq:3.1}),  that   gp($\tilde{g}_{j}': j \in J$)(1) $= G',$  therefore gp($\tilde{g}_{j}': j \in J$) $\simeq G'.$
Obviously,  $\tilde{g}_j'\in H$ for every $j\in J.$

Secondly,  $d^{c_{\iota (i)}}(1) = a_i$ for each $i\in I.$  Denote $\tilde{a}_i = d^{c_{\iota (i)}}$ for $i\in I$. We set  $\tilde{G}=$ gp($\tilde{a}_i, \tilde{g}_j': i \in I, j\in J$). Then $\tilde{G}(1)=G.$ We must show that $\tilde{G}\simeq G(1)$, and  conclude that $\tilde{G}\simeq G.$ 

Obviously, there is a natural homomorphism $\mu : \tilde{G}\rightarrow \tilde{G}(1)$ with the image $\tilde{G}(1).$  Obviously,  the restriction of $\mu $ to gp($\tilde{g}_{j}': j \in J$) is an isomorphism of $\tilde{G}'$ onto $G'.$ Suppose that  
$z=z(\tilde{a}_{i_1}, ..., \tilde{a}_{i_{k}}, \tilde{g}_{j_1}', ..., \tilde{g}_{j_q}')\in $  
ker ($\mu$). Then the sum of exponents $\sigma_{i_t}$ of $\tilde{a}_{i_t}$ in $z$ is $0$ for each $t = 1, ..., k.$ Since the sequence $s_1, s_2, ..., $ is strictly  uneven
sparse all other nontrivial values of   $\tilde{a}_{i_t}$ corresponding to the its occurs in $z$ are 
$\sigma_{i_t}$-exponents of the corresponding values of $\tilde{a}_{i_t},$ therefore are  trivial. This values don't depend from other factors of $z.$ 
Thus, $z \in \tilde{G}'$ and therefore $z=1.$ Hence $\mu$ is an isomorphism, and $G$ embeds in $H$.

 Theorem is proved. 

\paragraph{\bf Proof of    Corollary \ref{co:2}.} Let $G_{ab}= \bar{A} \times \bar{T},$ where $\bar{A}$ is a free abelian group   with  basis $\{\bar{a}_i: i \in I\}$, as before, and $\bar{T}$ is a finite abelian group. Let $a_i$ be a preimage of $\bar{a}_i$ in $G$ for $i\in I.$ We define $K=$ gp($a_i, G': i \in I$). This subgroup has a finite index in $G.$  Then we define elements $\tilde{a}_i, \tilde{g}_j'$ for $i \in I$ and $j\in J$ as above. We set $\tilde{G}_1=$ gp($\tilde{a}_i, \tilde{g}_j': i \in I, j\in J$).  Then $\tilde{G}_1(1)=K,$ and $\tilde{G}_1\simeq \tilde{G}_1(1)$. Hence $\tilde{G}_1\simeq K.$ This can be confirmed by the same argument as in  the proof of Theorem \ref{th:1}. 

Corollary is proved.

 \section{\bf Proof of   Theorem \ref{th:3}.}

First, we prove a number of auxiliary statements.

Let $G$ be a group and  $V = G\ Wr\ \mathbb{Z}$ be the Cartesian wreath product of $G$ and the  infinite cyclic group $\mathbb{Z} = $  gp($b$). Denote by $D_V$ the base group of $V.$ For any $u \in G$, let $u^{(0)}\in D_V$  be  the constant function $u^{(0)}: B \rightarrow G, u^{(0)}(b^i)=u, i \in \mathbb{Z}. $ Then $[u^{(0)}, b] = 1.$

\begin{lemma}
\label{le:7}
For any element $d \in D_V$  there is an element $x\in D_V$ for which 
\begin{equation}
\label{eq:4.1}
d = [x, b].
\end{equation}
Moreover, for any $u\in G$ there is a unique $x$ for which $x(1) = u.$
\end{lemma}
\begin{proof} 
Let $d = (...\ d_{-2}, d_{-1}, d_{0}, d_1, d_2, ... )$ and $x = (...\  x_{-2}, x_{-1}, x_{0}, x_1, x_2, ... ),$ where $d_i = d(b^i)$ and $x_i = x(b^i).$ Then (\ref{eq:4.1}) is equivalent to the system of equations 
\begin{equation}
\label{eq:4.2}
x_j^{-1}x_{j+1} = d_j, \  j\in \mathbb{Z}.
\end{equation} 
After  setting $x_0 = u, u \in U,$ we uniquely compute for $i\geq 1$, that 
\begin{equation}
\label{eq:4.3}
x_i =  u d_0 ... d_{i-1}\  {\rm and}\  x_{-i} = u d_{-1}^{-1} ... d_{-i}^{-1}.
\end{equation} 

\end{proof}

In other words, $d$ is a discrete (right) derivative of $x$, and $x$ is a discrete integral of $d$. This integral is uniquely defined by $d$ and its value $x(1) = u$. We will denote it as  $I_1(d, u)$ and write  
$I_1(d, u)' = d.$ Then we define $I_2(d, u) = I_1(I_1(d, u), u) ...,$ and so on. For simplicity, we keep $u$ for all integrals.  

\begin{corollary}
\label{co:8}
For each $u\in G$,  there are  a series of elements  $u^{(k)}\in D_V, k = 1, 2, ..., $ for which 
\begin{equation}
\label{eq:4.4}
[u^{(k)}, b] = u^{(k-1)} \  {\rm and}\ u^{(k)}(1) = u.
\end{equation}
In particular, 
\begin{equation}
\label{eq:4.5}
[u^{(k)}, b; k] = u^{(0)}\ {\rm and} \ 
[u^{(i)}, b; k] =1\ {\rm for}\  i < k.
\end{equation} 
\end{corollary}
\begin{proof}
We define $u^{(k)} = I_k(u^{(0)}, u)$ for $k = 1, 2, ...$. 

\end{proof}

Now, let $G$ be a group and  $W = G\ Wr \ \mathbb{Z}^2$ be a Cartesian wreath product, where $\mathbb{Z}^2 =$ gp($b_1$) $\times $ gp($b_2$) is the free abelian group  of rank $2$ with basis $\{b_1, b_2\}.$  Let $D_{W}$ denote the base group of $W$ consisting of all functions  $f : \mathbb{Z}^2 \rightarrow G, \  f(i, j) = f(b_1^i, b_2^j) = u_{ij}, \ i, j \in \mathbb{Z}.$

For any $u\in G$ we consider the subgroup $D_{u} = $ gp($u$) $Wr$ $\mathbb{Z}^2$  of $W$ generated by $b_1, b_2$ and all functions $f : \mathbb{Z}^2 \rightarrow$ gp($u$) which make up 
the base group $D_{u} \leq D_W$. Let  $f(i, j) = f(b_1^i, b_2^j) = u^{\alpha_{ij}}, 
\alpha_{ij} \in \mathbb{Z}, \ i, j \in \mathbb{Z}.$ 
Then for  any $i_0\in \mathbb{Z}$,  the set $C(i_0)(K)$  of elements of the form   $K =  \{u^{\alpha_{i_0,  j}}, j \in \mathbb{Z}\}$,  will be called $i_0$-{\it column} of $D_u$, and for any $j_0 \in \mathbb{Z}$ the set $R(j_0)(L)$ of elements of the form  
$L = \{u^{\alpha_{i,  j_0}}, i \in \mathbb{Z}\}$  will be called $j_0$-{\it row} of $D_u$. The set $K$ can be considered as element of the base group $D_{2, u}$   of $V_{2, u} =$ gp($u$)$Wr$ gp($b_2$), and similarly the set $L$ can be treated as element of the base group $D_{1,u}$ of $V_{1, u} =$ gp($u$)$Wr$ gp($b_1$).

Let $C(i_0)(w^{(0)}),$ be a constant column corresponding to $w\in $ gp($u$). By  Corollary \ref{co:8}, we get a series of columns $C(i_0)(w^{(t)})$ such that $C(i_0)(w^{(t)})' = C(i_0)(w^{(t-1)})$ and $C(i_0)(w^{(t)})(i_0, 1) = w, \  t = 1, 2, ...$.  Then $[C(i_0)(w^{(t)}), b_2; t] = C(i_0)(w^{(0)})$ and 
$[C(i_0)(w^{(t)}), b_2; t + r] = 1$ for every $r \geq 1.$ 

Similarly, we get the elements 
$R(j_0)(w^{(t)}), \  t = 0, 1, 2, ...$, that  satisfy the following  properties: $[R(j_0)(w^{(t)}), b_1; t] = $
$R(j_0)(w^{(0)})$ and 
$[R(j_0)(w^{(t)}), b_1; t + r] = 1$ for every $r \geq 1.$

Now we are ready to prove a key lemma that allows us to distinguish individual elements of a given finite set, while at the same time making other elements of this set trivial. We are dealing with the group $W=
G\  Wr $(gp($b_1$) $\times $ gp($b_2$)) defined above. 

\begin{lemma}
\label{le:9}
Let $u_1, ..., u_t$ be a finite set of nontrivial elements of $G$. Let $u_i^{(0)}\in D_W$ be a constant function with the value $u_i$.  Then there exist  functions $f_i\in D_W,$ all of whose values  belong to gp($u_i$), which satisfy the following properties. 
\begin{equation}
\label{eq:4.6}
[f_i, b_2; i; b_1, t - i] = u_i^{(0)}, [f_j, b_2; i; b_1, t - i] = 1 \  {\rm for} \ i\not=j, \ i, j \in \{1, ..., t\}.
\end{equation}
\end{lemma}
\begin{proof}
To construct $f_i$ we define its $0$-row 
$$R(0)(u_i^{(t-i)}) = ( ... u_i^{\alpha_{-1, 0}},  u_i^{\alpha_{0, 0}}, u_i^{\alpha_{1, 0}}, ... ).$$  
Then we build each column as
$$C(j)(u_i^{\alpha_{j, 0}})^{( i)}$$
\noindent and we have as a result $f_i$. By construction every $j$th column of 
$[f_i, b_2; i]$ is a constant function with value  $u_i^{\alpha_{0, j}}, j\in \mathbb{Z}.$ Then  
$$[f_i, b_2; i, b_1; t-i] = u_i^{(0)}\  {\rm and }\ [f_q, b_2; i, b_1; t-i] =  1 \ {\rm for}\  q < i.$$
It happens because the columns are constructed as discrete integrals. 

Let $q > i.$ By construction $R_0(u_q^{(t-q)}, b_1; t-i] = 1.$
In other words, this  is true for a $0$-row that does not change during the process of differentiating columns.

Consider a more general case.
Assume that $u \in G$ and  $q \geq 0$. For $r\geq 0$, we fix $0$-row $R(0)(u^{(r)} = ( ... u^{\alpha_{0, -1}},  u^{\alpha_{0, 0}}, u^{\alpha_{0, 1}}, ... ).$.  Then we expand  this row to element $f \in D_V$  by adding the columns 
$C(j)(u^{\alpha_{j,0}})^{(s)}$  for some $s\geq 0$. Obviously, for $s=0$ we have $[f, b_1; r + 1] = 1.$ We will prove by induction on $s$ that this equality is true in general case.  

Let it is true for $s-1$.  The value $u^{\alpha_{ij}}$ of the function $[f, b_1; r + 1]$ in any point $(i, j)$ can be computed as follows. There is a $\mathbb{Z}$-linear function $L(\alpha_{i,j}, \alpha_{i+1,j}, ..., \alpha_{i+r+1,j}$.
This function does not depend from $s$. By our assumptions, the value of this function for any $s$ is $0$ for $j=0$ and any $i.$ By the assumption of induction, for $s-1$ the value of this function is zero for every $j$.

Then this is true for $j = 1$ and $j=-1$. Indeed, if the $1$-row for $s$ is 
$$( ..., u^{\beta_{-1, 1}}, u^{\beta_{0, 1}}, u^{\beta_{1, 1}}, ...)$$
\noindent 
and $0$-row is  
$$( ... u^{\gamma_{-1, 1}}, u^{\gamma_{0, 1}}, u^{\gamma_{1, 1}} ...),$$ 
\noindent
then $0$-row for $s-1$ is 
$$( ..., u^{\beta_{-1, 1}- \gamma_{-1, 1}}, u^{\beta_{0, 1} - \gamma_{0,1}}, u^{\beta_{1, 1} - \gamma_{1,1}}, ...).$$

Then by the assumption of induction
$$L(\beta_{i,j} - \gamma_{i,j}) =  L(\beta_{i,j})  - L(\gamma_{i,j}) = - L(\gamma_{i,j} = 0.$$ Similarly, this can be proved for $j = -1$. 
Continuing, we will get that this is true for $s$ and each $j$.

\end{proof}

 We proceed directly to the proof of the Theorem \ref{th:3}. 
 
 Let  $G$ be a countable  group  such that the abelianization $G_{ab}$ is a direct product of a free abelian group  $\bar{A}$ with a basis $\{\bar{a}_i : i \in I \subseteq \mathbb{N}\}$ and  a finite abelian group $\bar{U}=$ gp($\bar{u}_1, ..., \bar{u}_t$).  Let $a_i$ denote a preimage of $\bar{a}_i$ and $u_j$ denote a preimage of $\bar{u_j}$ in $G$.
 Let $A=$ gp($a_i: i \in I$) and $U=$ gp($u_1, ..., u_t$). 
 
 Consider the Cartesian wreath product  $\tilde{W} = G\ Wr\ (C \times B)$, where $C$ = gp($c$) is an infinite cyclic group and $B$ is  a free abelian group with base $\{b_1, b_2\}.$ Then $\tilde{W} \simeq G\ Wr\ \mathbb{Z}^3.$
 By $D_{\tilde{W}}$ we denote the base group of the group $\tilde{W}.$
 
 First, we will do the same as in the proof of  Theorem \ref{th:1}. Let $ s_1 < ... < s_i < ... $ be a  strictly  uneven
sparse sequence of positive integers, i.e., $s_{i+j}-s_j=s_{k+l}-s_l$ if and only if $i=k$ and $j=l.$ For definiteness, we  take  $s_{i}= 2^i$ for $i= 1, 2, ... .$ 
For brevity,  denote $c_i=c^{s_i}, i =  1, 2, ... ,$  and  set  
${\mathcal C} = \{c_i: i =  1, 2, ...\}.$

Suppose that $G$ is generated by a set of  elements $\{ g_m: m\in M \subset \mathbb{N}\}$ such that   $G'$ = gp($g_{m,m'} = [g_m, g_m'] : m, m'\in M$).  Then $G=$ gp($a_i, u_{j},  g_{mm'}: {\rm for}\ i \in I,  j\in \{1, ..., t\}, m, m'\in M$).

Let $H=$ gp($c, b_1, b_2, d$) where $d\in D_{\tilde{W}}$ is defined as follows.

Let  ${\mathcal C}_1\subseteq {\mathcal C}$ be a disjoint union  $\{c_1, ..., c_t\} \sqcup {\mathcal I} \sqcup  {\mathcal M}$ where $|{\mathcal I}|=|I|,  |{\mathcal M}|= |M|$.
 Here  $I$ is in one-to-one correspondence $\iota$ with the set of indexes of elements in ${\mathcal I}$, $M$ is in one-to-one correspondence $\mu$ with the set of  indexes  of elements in ${\mathcal M}$.

Let    $D_W$ be the  base group in $W = G\ Wr\  {\rm gp(} (b_1{\rm )}\ \times {\rm gp)}{b_2}$) and  $u^{(0)} \in D_W$ denote a constant function with the value $u\in G.$ Let $f(j) \in D_W, j = 1, ..., t,$ are the elements  constructed in  Lemma \ref{le:9}. When constructing $d$ we use $f(j), j = 1, ..., t; a_i^{(0)}, i \in I,$ and $g_m^{(0)}, m \in M.$  
All values of $d$ belong to $D_W$. 

Then we set 

$$d(c_j) =  f(j),\  j = 1, ..., t;$$
\begin{equation}
\label{eq:4.7}
d(c_{\iota (i)}) = a_i^{(0)},\ i \in I;
\end{equation}
$$ d(c_{\mu (m)}) = g_m^{(0)},\ m \in M.$$
\noindent 
and we set $d(c^s) = 1$ in all other cases when $s\not\in {\mathcal C}_1.$    
 
Then $G^{(0)} =$ gp($u_j^{(0)}, a_i^{(0)}, g_m^{(0)} : j = 1, ..., t; i \in I, m \in M$) $\simeq G.$ For any 
$u^{(0)},$ one has $u^{(0)}, b_1] = u^{(0)}, b_2] = 1.$ We note also that $G^{(0)} =$ gp($u_j^{(0)}, a_i^{(0)}, g_{m,m'}^{(0)} : j = 1, ..., t; i \in I, m, m' \in M$) 

For any $h \in D, h(1)$ means $h(1, 1, 1)$. We  write $\tilde{h}(1)$ when all other values $h(c^i, b_1^j, b^k)$ are trivial.     

At first we prove by direct computation that $H$ contains all elements  of the form $\tilde{g}_{m,m'}^{(0)}(1)$ (remind that $c_i$ means $c^{s_i})$ :
\begin{equation}
\label{eq:4.8}
\tilde{g}_{m,m'}^{(0)}(1) = [d^{c_{\mu (m)}},  d^{c_{\mu (m')}}], \  m, m' \in M.
\end{equation}  
Similarly we get
\begin{equation}
\label{eq:4.9}
\tilde{u}_j^{(0)}(1) = [d, b_2; j, b_1, t - j]^{c_j}, \   j = 1, 2, ... , t.
\end{equation}
If $|I| \geq 2$, we cannot get $\tilde{a}_i^{(0)}$ in the similar way. Instead we will use the following elements: 
\begin{equation}
\label{eq:4.10}
\bar{a}_i^{(0)}(1) = (d^{c_{\iota (i)}}, \  i \in I). 
\end{equation}
We have
\begin{equation}
\label{eq:4.11}
\tilde{G} = gp(\tilde{u}_j^{(0)}(1), \tilde{g}_{m,m'}^{(0)}(1), a_i^{(0)}(1) : j = 1, ..., t; m, m'\in M, i \in I) 
\simeq G. 
\end{equation}
 Let 
 \begin{equation}
 \label{eq:4.12}
 \bar{G} = {\rm gp}(\tilde{u}_j^{(0)}(1), \tilde{g}_{m,m'}^{(0)}(1), \bar{a}_i^{(0)}(1) : j = 1, ..., t; m, m'\in M, i \in I) \leq H.
 \end{equation}
There is a natural homomorphism (projection) $\nu$ of $\bar{G}$ onto $\tilde{G}.$ In fact, $\nu$ is an isomorphism. Indeed, suppose that for some word $z$ we have 
\begin{equation}
\label{eq:4.13}
 z(\bar{a}_1^{(0)}(1), ..., \bar{a}_k^{(0)}(1), \tilde{u}_1^{(0)}(1), ..., \tilde{u}_t^{(0)}(1), \tilde{g}_{m_1.m_1'}^{(0)}(1), ..., \tilde{g}_{m_q,m_q'}^{(0)}(1)) \in  {\rm ker}(\mu ).
 \end{equation} Since $a_1, ..., a_k$ induce a part of base of the free abelian group $G_{ab}$ every  exponent sum $\sigma_i$ of $\tilde{a}_i^{(0)}(1), i = 1, ..., k, $ in $z$ is $0.$  Then every other value corresponding to entries of $\tilde{a}_i$ in $z$  is trivial. Then $g$ is independent  of $\bar{a}_i(1)$ for each $i\in I.$ Therefore, $g$ has only trivial values outside of $1$. It follows that $\bar{G} \simeq G.$

Theorem is proved. 

\begin{remark}
\label{re:10}
If the group G is finitely generated, then the proof of  Theorem \ref{th:3}\ can be carried out without introducing elements of the form $a_i, i \in I.$
\end{remark}
 
\section{Proof of Theorem \ref{th:5}} 

Let us see what can be said if a group $G$ is generated by a finite set of elements of finite orders.

 Let $G$ be a group and $V = G\ Wr\ \mathbb{Z}$ be a wreath product of $G$ and $\mathbb{Z} =$ gp($b$). As usual $D_V$ means the base group of $V$. 

\begin{lemma}
 \label{le:11}
Let $u\in G$ be an element of a finite order $l.$ Suppose, that every component of $d \in D_V$ belongs to 
{\rm gp(}$u${\rm )} and  $d$ is a non-constant periodic function. This means that there is a number $r > 0$ (period) such that 
$d(b^{i+r}) = d(b^i)$ for all $i.$  Let $d(1) = u^{i_0}, i_0 \in \mathbb{N}.$ 

Let $x \in D_V$ satisfy the condition $[x, b] = d$ and $x(1) = u^{i_0}.$  Then $x$ is periodic  with period $lr$.  
\end{lemma}
 \begin{proof}
Let $d = (... d_{-2}, d_{-1}, d_{0}, d_1, d_2, ... )$ and $x = (... x_{-2}, x_{-1}, x_{0}, x_1, x_2, ... ),$ where $d_i = d(b^i)$ and $x_i = x(b^i).$  We find $x$ as in the proof of  Lemma \ref{le:7}.

Then 
for any $i\geq 0$, by (\ref{eq:4.3}), 
$$x_{i +lr} = \prod_{j=0}^{lr-1}u_{i_0}d_j = u_{i_o}\ {\rm and}\  x_{-i - lr} = \prod_{j=-1}^{- lr}u_{i_0}d_j^{-1} 
= u_{i_0},$$
\noindent
because, the product of all elements from $l$ periods is $1$. 
 Therefore $x(i+r) = x(i)$ for all $i.$ 
  
 \end{proof}
 \begin{corollary}
 \label{co:12}
 Suppose that the conditions of  Lemma \ref{le:11}  are satisfied. Let $V_{lr} = G\ wr\ \mathbb{Z}_{lr}$ be the wreath product of $G$ with the cyclic group $\mathbb{Z}_{lr} =$ {\rm gp(}$b_{lr}${\rm )} of order $lr$. By $D_{V_{lr}}$ we denote the base subgroup of $V_{lr}$. Since $d$ and $x$ are $lr$-periodic they can be considered as elements of $D_{V_{lr}}$. Then $d = [x, b_{lr}]$   in $V_{lr}$.  
 \end{corollary}

 Let  $u\in G, u\not=1$ and $u^l = 1.$ Then 
 $$u^{(1)} = ( ... u^{-2}, u^{-1}, u^{0}, u^{1}, u^{2}, ... )$$ 
\noindent is obviously $l$-periodic.  By   Lemma \ref{le:11}, $u^{(2)}$ is  $l^2$-periodic, and so on. 
 
 We can consider the elements $u^{(0)}, u^{(1)}, ..., u^{(t)}$ as elements of $D_{V_{l^{t}}}$, the base subgroup of  $V_{l^{t}} = G\ Wr\ \mathbb{Z}_{l^{t}},$ where $\mathbb{Z}_{l^{t}} =$   gp($b_{l^{t}}$) is the cyclic group of order $l^{t}.$

It follows that we have the following finite analogue of   Corollary \ref{co:8}.
 
 \begin{corollary}
 \label{co:13} For each $u\in G, u^l = 1,$  there are  a series of elements  $u^{(k)}\in D_{V_{l^{t}}}, k = 1, 2, ..., t $ for which 
\begin{equation}
\label{eq:5.1}
[u^{(k)}, b] = u^{(k-1)} \  {\rm and}\ u^{(k)}(1) = u.
\end{equation}
In particular, 
\begin{equation}
\label{eq:5.2} 
[u^{(k)}, b; k] = u^{(0)}\ {\rm and} \ 
[u^{(i)}, b; k] =1\ {\rm for}\  i < k.
\end{equation} 
\end{corollary}
 
\begin{lemma}
\label{le:14}
Let $u_1, ..., u_m \in G$ be  elements of  finite orders $l_1, ..., l_m,$ respectively. Let 
$s = $  {\rm lcm(}$l_1^{m}, ..., l_m^{m}${\rm )}. Let
$V_{s} = G\ wr\ \mathbb{Z}_{s}$ be the wreath product of $G$ with the cyclic group $\mathbb{Z}_{s} =$  {\rm gp(}$b_s{\rm )}$   of order $s$. By $D_{V_s}$ we denote the base subgroup of $V_{s}$. Then each of the elements $u_i^{(i)}, i = 1, ..., m$ that can be constructed by   Corollary  \ref{co:13}  can be considered as element of $D_{V_s}$. These elements have the following properties: 
 \begin{equation}
 \label{eq:5.3} [u_i^{(k)}, b_s] = u_i^{(k-1)}, [u_i^{(k)}, b_s; k] = u_i^{(0)}
 \ {\rm for}\ k = 1, ..., m
 \end{equation}
\noindent
and
\begin{equation}
\label{eq:5.4}[u^{(i)}, b_s; k] =1\ {\rm for}\  i < k.
\end{equation} 
\end{lemma}

The following lemma is an analogue of Lemma \ref{le:9}. 

\begin{lemma}
\label{le:15}
Let $u_1, ..., u_m \in G$ be  elements of  finite orders $l_1, ..., l_m,$ respectively. Let 
$s = $  {\rm lcm(}$l_1^{m}, ..., l_m^{m}${\rm )}. Let
$W_{s} = G\ wr\ \mathbb{Z}_{s}^2$ be the wreath product of $G$ with the direct product   {\rm gp(}$b_{1,s}${\rm )}$ \times $ {\rm gp(}$b_{2,s}${\rm )}   of two cyclic groups of order $s$ each.  Let $u_i^{(0)}\in D_{W_s}$  be a constant function with the value $u_i$.  Then there exist  functions $f_i\in D_{W_s},$ all of whose values  belong to {\rm gp(}$u_i${\rm )}, which satisfy the following properties. 
\begin{equation}
\label{eq:5.5}
[f_i, b_{2,s}; i; b_{1,s}; t - i] = u_i^{(0)}, [f_j, b_{2,s}; i; b_{1,s}; t - i] = 1 \  {\rm for} \ i\not=j, \ i, j \in \{1, ..., m\}.
\end{equation}
\end{lemma}

The proof completely repeats the proof of  Lemma \ref{le:9}, taking into account  Lemma \ref{le:14}.

We proceed directly to the proof of Theorem \ref{th:5}. We keep the notation introduced above.

Now $G$  is a  group generated by a finite set $u_1, ..., u_m$ of elements of finite orders $l_1, ..., l_m$, respectively.  We can assume that $m\geq 2.$ By Lemma \ref{le:15},  we construct the wreath product $W_{s} = G\ wr\ \mathbb{Z}_{s}^2$ and elements $f_i\in D_{V_s}, i = 1, ..., m,$ which satisfy the equalities (\ref{eq:5.5}).   
 
Let $\tilde{W}_s = G\ wr \ (\mathbb{Z}_{s^2} \times \mathbb{Z}_s^2),$ where $\mathbb{Z}_{s^2} = $ gp($c_s$) and $\mathbb{Z}_s^2 =$ gp($b_{1s}$) $\times $ gp($b_{2s}$). Let $s_1 < s_2 < ... < s_m $ be a strictly  uneven
sparse sequence of positive integers such that $s_{i+j}-s_j=s_{k+l}-s_l$ if and only if $i=k$ and $j=l.$ For definiteness, we  take  $s_{i}= 2^i$ for $i= 1, ... m.$ Since $s \geq 2^{m}$ this property is valid modulo $s^2\geq 2^{m+1}$.   

The rest of the proof completely repeats the arguments of the proof of Theorem \ref{th:3}.
 
Theorem is proved. 

\paragraph{Proof of Corollary \ref{co:6}.}

Now we can take in the proof of Theorem \ref{th:5}, instead of the active group  $\mathbb{Z}_{s^2} \times \mathbb{Z}_s^2$ in $\tilde{W}_s,$  the group $\mathbb{Z}_{e^{m+1}}\times \mathbb{Z}_e^2,$ and get the statement 1). The statement 2) follows directly from 1).

Corollary is proved.

Additional information

\bigskip
{\bf Vitaly Roman'kov}

Affiliation: Sobolev Institute of Mathematics, Siberian Branch of the Russian Academy of Sciences,
Omsk Division, Pevtsova street 13, 644099, Omsk, Russia.

Dostoevsky Omsk State University, Mira 55-a, 644077, Omsk, Russia.

E-mail: romankov48@mail.ru

\end{document}